\documentclass{amsart}
\usepackage{amssymb}
\usepackage[mathcal]{eucal}
\usepackage{upref}

\usepackage[matrix,arrow,curve,cmtip]{xy}
\let\objectstyle=\displaystyle

\DeclareMathAlphabet{\EuRm}{U}{eur}{m}{n}
\SetMathAlphabet{\EuRm}{bold}{U}{eur}{b}{n}

\makeatletter
\let\c@equation\c@subsection

\makeatother

\theoremstyle{plain}   

\newtheoremstyle{slplain}
  {.5\baselineskip\@plus.2\baselineskip\@minus.2\baselineskip}
  {.5\baselineskip\@plus.2\baselineskip\@minus.2\baselineskip}
  {\slshape}
  {}
  {\bfseries}
  {.}
  { }
  {}

\theoremstyle{slplain}
\newtheorem{thm}[equation]{Theorem}  
     
\newtheorem{lem}[equation]{Lemma}         
\newtheorem{prop}[equation]{Proposition}

\theoremstyle{definition}
\newtheorem{defn}[equation]{Definition} 

\theoremstyle{remark}

\newcommand{\thmref}{Theorem~\ref}
\newcommand{\propref}{Proposition~\ref}

\newcommand{\lemref}{Lemma~\ref}
\newcommand{\defref}{Definition~\ref}

    {%
     \begin{description}%
    }%
    {\end{description}}

\newcommand{\clearemptydoublepage}%
  {\newpage
    {\pagestyle{empty}%
      \cleardoublepage
    } }
\newcommand{\Period}{\rlap{\enspace .}}

\newcommand{\suchthat}{\bigm|}

\newcommand{\smallcoprod}{\mathchoice{\amalg}%
               {\amalg}%
               {{\scriptscriptstyle{\amalg}}}%
               {{\scriptscriptstyle{\amalg}}}}

\newcommand{\undercat}[2]{%
  {(\mathord{#2}\nonscript\,\mathord\downarrow\nonscript\,\mathord{#1})}}

\newcommand{\overcat}[2]{%
  {(\mathord{#1}\nonscript\,\mathord\downarrow\nonscript\,\mathord{#2})}}

\newcommand{\adj}[4]{#1\negmedspace: #2\rightleftarrows #3:\negmedspace #4}

\newcommand{\cat}[1]{\mathcal{#1}}

\newcommand{\Id}{1}

\newcommand{\fakewidth}[2]{%
  \begingroup
    \setbox0 =\hbox{#2}%
    \hbox to \wd0{\hss #1\hss}%
  \endgroup
}

\newcommand{\F}{\mathrm{F}}

\newcommand{\G}{\mathrm{G}}

\newcommand{\U}{\mathrm{U}}

\newcommand{\mapunder}[3]{
  \vcenter{
    \def\objectstyle{\scriptstyle}
    \def\labelstyle{\scriptstyle}
    \xymatrix@C=0em@R=1ex{
      & {#1} \ar[dl] \ar[dr]\\
      {#2} \ar[rr]
      && {#3}
      }
    }
}
\newcommand{\labelmapunder}[4]{
  \vcenter{
    \def\objectstyle{\scriptstyle}
    \def\labelstyle{\scriptstyle}
    \xymatrix@C=0em@R=1ex{
      & {#1} \ar[dl] \ar[dr]\\
      {#2} \ar[rr]_-{#4}
      && {#3}
      }
    }
}
\newcommand{\mapover}[3]{
  \vcenter{
    \def\objectstyle{\scriptstyle}
    \def\labelstyle{\scriptstyle}
      \xymatrix@C=0em@R=1ex{
        {#1} \ar[rr] \ar[dr]
        && {#2} \ar[dl]\\
        & {#3}
      }
    }
}
\newcommand{\labelmapover}[4]{
  \vcenter{
    \def\objectstyle{\scriptstyle}
    \def\labelstyle{\scriptstyle}
      \xymatrix@C=0em@R=1ex{
        {#1} \ar[rr]^-{#4} \ar[dr]
        && {#2} \ar[dl]\\
        & {#3}
      }
    }
}

\begin{document}

\title[Overcategories and undercategories of model
categories]{Overcategories and undercategories\\of model categories} 

\author{Philip S. Hirschhorn}

\address{Department of Mathematics\\
   Wellesley College\\
   Wellesley, Massachusetts 02481}

\email{psh@math.mit.edu}

\urladdr{http://www-math.mit.edu/\textasciitilde psh}

\date{June 11, 2005}

\maketitle

If $\cat M$ is a model category and $Z$ is an object of $\cat M$, then
there are model category structures on the categories $\overcat{\cat
  M}{Z}$ (the category of objects of $\cat M$ over $Z$) and
$\undercat{\cat M}{Z}$ (the category of objects of $\cat M$ under $Z$)
under which a map is a cofibration, fibration, or weak equivalence if
and only if its image in $\cat M$ under the forgetful functor is,
respectively, a cofibration, fibration, or weak equivalence.  It is
asserted without proof in \cite{MCL} that if $\cat M$ is cofibrantly
generated, cellular, or proper, then so is the overcategory
$\overcat{\cat M}{Z}$.  The purpose of this note is to fill in the
proofs of those assertions (see \thmref{thm:Over}) and to state and
prove the analogous results for undercategories (see
\thmref{thm:Under}).

\section{Overcategories}
\label{sec:Overcat}

\begin{defn}
  \label{def:OvCat}
  If $\cat M$ is a category and $Z$ is an object of $\cat M$, then the
  category $\overcat{\cat M}{Z}$ of \emph{objects of $\cat M$ over
    $Z$} is the category in which
  \begin{itemize}
  \item an object is a map $X \to Z$ in $\cat M$,
  \item a map from $X \to Z$ to $Y \to Z$ is a map $X \to Y$ in $\cat
    M$ such that the triangle
    \begin{displaymath}
      \xymatrix@=1em{
        {X} \ar[rr] \ar[dr]
        && {Y} \ar[dl]\\
        & {Z}
      }
    \end{displaymath}
    commutes, and
  \item composition of maps is defined by composition of maps in $\cat
    M$.
  \end{itemize}
\end{defn}

\begin{defn}
  \label{def:forgetful}
  If $\cat M$ is a category and $Z$ is an object of $\cat M$, then the
  \emph{forgetful functor} $\G\colon \overcat{\cat M}{Z} \to \cat M$
  is the functor that takes the object $A \to Z$ of $\overcat{\cat
    M}{Z}$ to the object $A$ of $\cat M$ and the map
  $\mapover{A}{B}{Z}$ of $\overcat{\cat M}{Z}$ to the map $A \to B$ of
  $\cat M$.
\end{defn}

\begin{lem}
  \label{lem:SamePush}
  Let $\cat M$ be a cocomplete and complete category and let $Z$ be an
  object of $\cat M$.
  \begin{enumerate}
  \item The pushout in $\overcat{\cat M}{Z}$ of the diagram
    \begin{displaymath}
      \xymatrix{
        {C} \ar[dr]
        & {A} \ar[l] \ar[d] \ar[r]
        & {B} \ar[dl]\\
        & {Z}
      }
    \end{displaymath}
    is $P \to Z$ where $P$ is the pushout in $\cat M$ of the diagram
    \begin{displaymath}
      \xymatrix{
        {C}
        & {A} \ar[l] \ar[r]
        & {B}
      }
    \end{displaymath}
    and the structure map $P \to Z$ is the natural map from the
    pushout in $\cat M$.
  \item The pullback in $\overcat{\cat M}{Z}$ of the diagram
    \begin{displaymath}
      \xymatrix{
        {X} \ar[r] \ar[dr]
        & {Y} \ar[d]
        & {W} \ar[l] \ar[dl]\\
        & {Z}
      }
    \end{displaymath}
    is $P \to Z$ where $P$ is the pullback in $\cat M$ of the diagram
    \begin{displaymath}
      \xymatrix{
        {X} \ar[r]
        & {Y}
        & {W} \ar[l]
      }
    \end{displaymath}
    and the structure map $P \to Z$ is the composition $P \to Y \to
    Z$.
  \end{enumerate}
\end{lem}

\begin{proof}
  The described constructions possess the universal mapping properties
  that characterize the pushout (or pullback) in $\overcat{\cat
    M}{Z}$.
\end{proof}

\begin{lem}
  \label{lem:RelCelOver}
  Let $\cat M$ be a model category and let $Z$ be an object of $\cat
  M$.  If $S$ is a set of maps in $\cat M$ and $S_{Z}$ is the set of
  maps in $\overcat{\cat M}{Z}$ of the form
  \begin{displaymath}
    \xymatrix@=1em{
      {A} \ar[rr] \ar[dr]
      && {B} \ar[dl]\\
      & {Z}
    }
  \end{displaymath}
  in which the map $A \to B$ is an element of $S$, then a map
  $\mapover{X}{Y}{Z}$ in $\overcat{\cat M}{Z}$ is a relative
  $S_{Z}$-cell complex (see \cite[Definition~10.5.8]{MCL}) if and only
  if the map $X \to Y$ in $\cat M$ is a relative $S$-cell complex.
\end{lem}

\begin{proof}
  This follows from \lemref{lem:SamePush}.
\end{proof}

\begin{thm}
  \label{thm:OvCtCG}
  Let $\cat M$ be a cofibrantly generated model category (see
  \cite[Definition~11.1.2]{MCL}) with generating cofibrations $I$ and
  generating trivial cofibration $J$, and let $Z$ be an object of
  $\cat M$.  If
  \begin{enumerate}
  \item $I_{Z}$ is the set of maps in $\overcat{\cat M}{Z}$ of the form
    \begin{equation}
      \label{eq:OvCtCG}
      \xymatrix@=1em{
        {A} \ar[rr] \ar[dr]
        && {B} \ar[dl]\\
        & {Z}
      }
    \end{equation}
    in which the map $A \to B$ is an element of $I$ and
  \item $J_{Z}$ is the set of maps in $\overcat{\cat M}{Z}$ of the
    form \eqref{eq:OvCtCG} in which the map $A \to B$ is an element of
    $J$,
  \end{enumerate}
  then the standard model category structure on $\overcat{\cat M}{Z}$
  (in which a map $\mapover{X}{Y}{Z}$ is a cofibration, fibration, or
  weak equivalence in $\overcat{\cat M}{Z}$ if and only if the map $X
  \to Y$ is, respectively, a cofibration, fibration, or weak
  equivalence in $\cat M$) is cofibrantly generated, with generating
  cofibrations $I_{Z}$ and generating trivial cofibrations $J_{Z}$.
\end{thm}

\begin{proof}
  We will show that the set $I_{Z}$ permits the small object argument
  and that a map is a trivial fibration if and only if it has the
  right lifting property with respect to $I_{Z}$; the proof of the
  analogous statement for $J_{Z}$ is similar.
  
  \lemref{lem:SamePush} implies that the forgetful functor $\G\colon
  \overcat{\cat M}{Z} \to \cat M$ (see \defref{def:forgetful}) takes a
  relative $I_{Z}$-cell complex in $\overcat{\cat M}{Z}$ to a relative
  $I$-cell complex in $\cat M$, and so the set $I_{Z}$ permits the
  small object argument.
  
  Since every element of $I_{Z}$ is a cofibration in $\overcat{\cat
    M}{Z}$, every trivial fibration in $\overcat{\cat M}{Z}$ has the
  right lifting property with respect to every element of $I_{Z}$.  To
  show that every map with the right lifting property with respect to
  $I_{Z}$ is a trivial fibration, it is sufficient to show that every
  cofibration is a retract of a relative $I_{Z}$-cell complex (see
  \cite[Proposition~10.3.2]{MCL}).  Let $\mapover{X}{Y}{Z}$ be a
  cofibration in $\overcat{\cat M}{Z}$; then the map $X \to Y$ is a
  cofibration in $\cat M$, and we can factor it as $X \to W \to Y$ in
  $\cat M$ where $X \to W$ is a relative $I$-cell complex and $W \to
  Y$ is a trivial fibration.  Since $\mapover{W}{Y}{Z}$ is a trivial
  fibration in $\undercat{\cat M}{Z}$, the retract argument
  (\cite[Proposition~7.2.2]{MCL}) now implies that $\mapover{X}{Y}{Z}$
  is a retract of $\mapover{X}{W}{Z}$, and \lemref{lem:RelCelOver}
  implies that $\mapover{X}{W}{Z}$ is a relative $I_{Z}$-cell complex.
\end{proof}

\begin{thm}
  \label{thm:Over}
  Let $\cat M$ be a model category and let $Z$ be an object of $\cat
  M$.
  \begin{enumerate}
  \item If $\cat M$ is cofibrantly generated, then so is
    $\overcat{\cat M}{Z}$.
  \item If $\cat M$ is cellular, then so is $\overcat{\cat M}{Z}$.
  \item If $\cat M$ is left proper, right proper, or proper, then so
    is $\overcat{\cat M}{Z}$.
  \end{enumerate}
\end{thm}

\begin{proof}
  Part~1 follows from \thmref{thm:OvCtCG}, part~2 follows from
  \thmref{thm:OvCtCG} and \lemref{lem:RelCelOver}, and part~3 follows
  from \lemref{lem:SamePush}.
\end{proof}

\section{Undercategories}
\label{sec:UnderCat}

\begin{defn}
  \label{def:UndCat}
  If $\cat M$ is a category and $Z$ is an object of $\cat M$, then the
  category $\undercat{\cat M}{Z}$ of \emph{objects of $\cat M$ under
    $Z$} is the category in which
  \begin{itemize}
  \item an object is a map $Z \to X$ in $\cat M$,
  \item a map from $Z \to X$ to $Z \to Y$ is a map $X \to Y$ in $\cat
    M$ such that the triangle
    \begin{displaymath}
      \xymatrix@=1em{
        & {Z} \ar[dl] \ar[dr]\\
        {X} \ar[rr]
        && {Y}
      }
    \end{displaymath}
    commutes, and
  \item composition of maps is defined by composition of maps in $\cat
    M$.
  \end{itemize}
\end{defn}

\begin{prop}
  \label{prop:UndAdj}
  If $\cat M$ is a cocomplete category and $Z$ is an object of $\cat
  M$, then the forgetful functor $\U\colon \undercat{\cat M}{Z} \to
  \cat M$ that takes the object $Z \to Y$ to $Y$ is right adjoint to
  the functor $\F\colon \cat M \to \undercat{\cat M}{Z}$ that takes
  the object $X$ of $\cat M$ to $Z \to Z\smallcoprod X$ (where that
  structure map is the natural injection into the coproduct).
\end{prop}

\begin{proof}
  If $X$ is an object of $\cat M$ and $Z \to Y$ is an object of
  $\undercat{\cat M}{Z}$, then the universal mapping property of the
  coproduct implies that a map $\mapunder{Z}{Z\smallcoprod X}{Y}$ in
  $\undercat{\cat M}{Z}$ is entirely determined by the choice of a map
  $X \to Y$ in $\cat M$.
\end{proof}

\begin{lem}
  \label{lem:PushUnder}
  Let $\cat M$ be a cocomplete and complete category and let $Z$ be an
  object of $\cat M$.
  \begin{enumerate}
  \item The pushout in $\undercat{\cat M}{Z}$ of the diagram
    \begin{displaymath}
      \xymatrix{
        & {Z} \ar[dl] \ar[d] \ar[dr]\\
        {C}
        & {A} \ar[l] \ar[r]
        & {B}
      }
    \end{displaymath}
    is $Z \to P$ where $P$ is the pushout in $\cat M$ of the diagram
    \begin{displaymath}
      \xymatrix{
        {C}
        & {A} \ar[l] \ar[r]
        & {B}
      }
    \end{displaymath}
    and the structure map $Z \to P$ is the composition $Z \to A \to
    P$.
  \item The pullback in $\undercat{\cat M}{Z}$ of the diagram
    \begin{displaymath}
      \xymatrix{
        & {Z} \ar[dl] \ar[d] \ar[dr]\\
        {X} \ar[r]
        & {Y}
        & {W} \ar[l]
      }
    \end{displaymath}
    is $Z \to P$ where $P$ is the pullback in $\cat M$ of the diagram
    \begin{displaymath}
      \xymatrix{
        {X} \ar[r]
        & {Y}
        & {W} \ar[l]
      }
    \end{displaymath}
    and the structure map $Z \to P$ is the natural map to the pullback
    in $\cat M$.
  \end{enumerate}
\end{lem}

\begin{proof}
  The described constructions possess the universal mapping properties
  that characterize the pushout (or pullback) in $\undercat{\cat
    M}{Z}$.
\end{proof}

\begin{prop}
  \label{prop:UndPush}
  Let $\cat M$ be a cocomplete category, let $Z$ be an object of $\cat
  M$, and let $\adj{\F}{\cat M}{\undercat{\cat M}{Z}}{\U}$ be the
  adjoint pair of \propref{prop:UndAdj}.  If $f\colon A \to B$ is a
  map in $\cat M$ and 
  \begin{displaymath}
    \xymatrix{
      & {Z} \ar[dl] \ar[dr]^{i_{X}}\\
      {Z\smallcoprod A} \ar[rr]_-{i_{X}\smallcoprod g}
      && {X}
    }
  \end{displaymath}
  is a map in $\undercat{\cat M}{Z}$, then the pushout in
  $\undercat{\cat M}{Z}$ of the diagram
  \begin{displaymath}
    \xymatrix{
      & {Z} \ar[dl]_{i_{X}} \ar[d] \ar[dr]\\
      {X}
      & {Z\smallcoprod A} \ar[l]^-{i_{X}\smallcoprod g}
      \ar[r]_-{\Id_{Z}\smallcoprod f}
      & {Z\smallcoprod B}
    }
  \end{displaymath}
  is $Z \to P$ where $P$ is the pushout in $\cat M$ of the diagram
  \begin{displaymath}
    \xymatrix{
      {X}
      & {A} \ar[l]_{g} \ar[r]^{f}
      & {B}
    }
  \end{displaymath}
  and the structure map $Z \to P$ is the composition $Z
  \xrightarrow{i_{X}} X \to P$.
\end{prop}

\begin{proof}
  The described construction possesses the universal mapping property
  required of the pushout in $\undercat{\cat M}{Z}$.
\end{proof}

\begin{prop}
  \label{prop:RelCelUnd}
  Let $\cat M$ be a cocomplete category, let $Z$ be an object of $\cat
  M$, and let $\adj{\F}{\cat M}{\undercat{\cat M}{Z}}{\U}$ be the
  adjoint pair of \propref{prop:UndAdj}.  If $S$ is a set of maps in
  $\cat M$, then a relative $\F S$-cell complex (see
  \cite[Definition~10.5.8]{MCL}) $\mapunder{Z}{X}{Y}$ in
  $\undercat{\cat M}{Z}$ is a relative $S$-cell complex $X \to Y$ in
  $\cat M$ with structure maps defined by composition with the
  structure map of $Z \to X$.
\end{prop}

\begin{proof}
  This follows from \propref{prop:UndPush}.
\end{proof}

\begin{prop}
  \label{prop:UndSmOb}
  Let $\cat M$ be a cocomplete category, let $Z$ be an object of $\cat
  M$, and let $\adj{\F}{\cat M}{\undercat{\cat M}{Z}}{\U}$ be the
  adjoint pair of \propref{prop:UndAdj}.  If $S$ is a set of maps in
  $\cat M$ that permits the small object argument (see
  \cite[Definition~10.5.15]{MCL}), then $\F S$ is a set of maps in
  $\undercat{\cat M}{Z}$ that permits the small object argument.
\end{prop}

\begin{proof}
  This follows from \propref{prop:RelCelUnd} and the adjointness of
  the functors $\F$ and $\U$.
\end{proof}

\begin{thm}
  \label{thm:UndCofGen}
  Let $\cat M$ be a cofibrantly generated model category (see
  \cite[Definition~11.1.2]{MCL}) with generating cofibrations $I$ and
  generating trivial cofibrations $J$.  If $Z$ is an object of $\cat
  M$ and $\adj{\F}{\cat M}{\undercat{\cat M}{Z}}{\U}$ is the adjoint
  pair of \propref{prop:UndAdj}, then the standard model category
  structure on $\undercat{\cat M}{Z}$ (in which a map
  $\mapunder{Z}{X}{Y}$ is a cofibration, fibration, or weak
  equivalence in $\undercat{\cat M}{Z}$ if and only if the map $X \to
  Y$ is, respectively, a cofibration, fibration, or weak equivalence
  in $\cat M$) is cofibrantly generated, with generating cofibrations
  \begin{displaymath}
    \F I =
    \left\{
    \mapunder{Z}{Z\smallcoprod X}{Z\smallcoprod Y}
    = \F(A \to B)
    \suchthat (A \to B) \in I\right\}
  \end{displaymath}
  and generating trivial cofibrations
  \begin{displaymath}
    \F J =
    \left\{
    \mapunder{Z}{Z\smallcoprod X}{Z\smallcoprod Y}
    = \F(A \to B)
    \suchthat (A \to B) \in J\right\}
  \Period
  \end{displaymath}
\end{thm}

\begin{proof}
  We will use \cite[Theorem~11.3.2]{MCL} to show that there is a
  cofibrantly generated model category structure on $\undercat{\cat
    M}{Z}$ with generating cofibrations $\F I$ and generating trivial
  cofibration $\F J$, after which we will show that this coincides
  with the standard model category structure on $\undercat{\cat
    M}{Z}$.

  To apply \cite[Theorem~11.3.2]{MCL}, we must show that
  \begin{enumerate}
  \item both of the sets $\F I$ and $\F J$ permit the small object
    argument, and
  \item $\U$ takes relative $\F J$-cell complexes in $\undercat{\cat
      M}{Z}$ to weak equivalences in $\cat M$.
  \end{enumerate}
  The first condition follows from \propref{prop:UndSmOb}, and the
  second condition follows from \propref{prop:RelCelUnd}, since a
  relative $J$-cell complex is a trivial cofibration in $\cat M$.
  
  Thus, $\F I$ and $\F J$ are the generating cofibrations and
  generating trivial cofibrations of some model category structure on
  $\undercat{\cat M}{Z}$.  To see that this is the standard one, we
  must show that a map in $\undercat{\cat M}{Z}$ is a cofibration,
  fibration, or weak equivalence if and only if its image under $\U$
  is, respectively, a cofibration, fibration, or weak equivalence in
  $\cat M$.  For the weak equivalences, this follows from
  \cite[Theorem~11.3.2]{MCL}.  Since the fibrations of $\undercat{\cat
    M}{Z}$ are the maps with the right lifting property with respect
  to every element of $\F J$, the adjointness of $\F$ and $\U$ implies
  that these are exactly the maps whose images under $\U$ have the
  right lifting property with respect to $J$, i.e., exactly the maps
  whose images under $\U$ are fibrations in $\cat M$.  Finally, since
  the fibrations and the weak equivalences of a model category
  structure determine the cofibrations, the two model category
  structures on $\undercat{\cat M}{Z}$ must have the same cofibrations
  as well.
\end{proof}

\begin{thm}
  \label{thm:Under}
  Let $\cat M$ be a model category and let $Z$ be an object of $\cat
  M$.
  \begin{enumerate}
  \item If $\cat M$ is cofibrantly generated, then so is
    $\undercat{\cat M}{Z}$.
  \item If $\cat M$ is cellular, then so is $\undercat{\cat M}{Z}$.
  \item If $\cat M$ is left proper, right proper, or proper, then so
    is $\undercat{\cat M}{Z}$.
  \end{enumerate}
\end{thm}

\begin{proof}
  Part~1 follows from \thmref{thm:UndCofGen}, part~2 follows from
  \thmref{thm:UndCofGen} and \propref{prop:RelCelUnd}, and part~3
  follows from \lemref{lem:PushUnder}.
\end{proof}



\end{document}